\documentclass{svproc}	   

\usepackage{url}

\usepackage{amsmath,amssymb}

\usepackage{mathptmx}       
\usepackage{helvet}         
\usepackage{courier}        
\usepackage{type1cm}        
\usepackage{makeidx}         
\usepackage{graphicx}        
\usepackage{multicol}        
\usepackage[bottom]{footmisc}

\usepackage{breakurl}

\usepackage{hyperref}	     

\usepackage{caption}	     

\newcommand{\raisecomma}{\raisebox{2pt}{$,$}}
\newcommand{\raisedot}{\raisebox{2pt}{$.$}}

\newcommand{\correct}{\bf}

\newcommand{\deriv}{{\rm D}}
\newcommand{\nonnegint}{{\mathbb N}}

\newcommand{\pilow}{\underline{\pi}}
\newcommand{\pihigh}{\overline{\pi}}

\newcommand{\AM}{\text{AM}}
\newcommand{\GM}{\text{GM}}
\newcommand{\AGM}{\text{AGM}}

\newcommand{\half}{{\textstyle \frac{1}{2}}}

\newcommand{\PAGM}{\emph{Pi and the AGM}}

\newcommand{\Halmos}{\hspace*{\fill}\qed}

\begin{document}
\mainmatter		

\title{The  Borwein Brothers, Pi and the AGM}
\titlerunning{The Borweins, $\pi$ and the AGM}
\author{Richard P. Brent}
\authorrunning{Richard Brent}
\tocauthor{Richard Brent}
\institute{Mathematical Sciences Institute,\\
Australian National University,\\ Canberra, ACT 2600, Australia\\
and\\
CARMA, University of Newcastle,\\ 
Callaghan, NSW 2308, Australia.\\
\email{JBCC@rpbrent.com}
}

\maketitle

\begin{center}
\emph{In fond memory of Jonathan M.\ Borwein 1951--2016}
\end{center}

\begin{abstract}
We consider some of Jonathan and Peter Borweins' contributions to the
high-precision computation of $\pi$ and the elementary functions, with
particular reference to their book
{\PAGM} (Wiley, 1987).
Here ``AGM'' is the
\emph{arithmetic-geometric mean} of Gauss and Legendre.
Because the AGM converges quadratically, it can be combined with
fast multiplication algorithms 
to give fast algorithms for the {$n$-bit} computation of
$\pi$, and more generally
the elementary functions. 
These algorithms run in ``almost linear'' time
$O(M(n)\log n)$, where
$M(n)$ 
is the time for $n$-bit multiplication.
We outline some of the results and algorithms given
in {\PAGM}, and present some related (but new) results.
In particular, we improve the published error bounds for some
quadratically and quartically convergent algorithms for $\pi$,
such as the Gauss-Legendre algorithm.
We show that an iteration of 
the Borwein-Borwein quartic algorithm for $\pi$ is
equivalent to two iterations of the Gauss-Legendre quadratic algorithm
for $\pi$, in the sense that they produce exactly 
the same sequence of approximations to $\pi$ if performed using
exact arithmetic.
\end{abstract}

\section{Introduction}		\label{sec:intro}

Jonathan Borwein was fascinated by the constant $\pi$, and gave many
stimulating talks on this topic. The slides for most of these
talks may be found on the memorial website~\cite{Jon-talks}.
In my talk~\cite{rpb-JBCC-talk} at the 
Jonathan Borwein Commemorative Conference
I discussed the reasons for this fascination.
In a nutshell, it is that theorems about $\pi$ are often just
the tips of ``mathematical icebergs''~-- much of interest
lies hidden beneath the surface.

This paper considers some of Jonathan and Peter Borweins'
contributions to the high-precision computation of $\pi$ and the
elementary functions $\log$, $\exp$, $\arctan$, $\sin$, etc.
The material is mainly drawn from their fascinating book 
\emph{Pi and the AGM}~\cite{PAGM}. We make no attempt to review
the whole book~-- a reader interested in the complete contents should
consult one of the reviews~\cite{Andrews90,Askey88,Berndt88,Wimp88}
or, better, read the book itself.
We do not try to distinguish between the contributions of
Jonathan and his brother Peter~-- so far as we know, they contributed
equally to the book, although no doubt in different ways.

We take the opportunity to present some new results that are related to
the material in {\PAGM}.  For example, the error after a finite number of
iterations of some of the quadratically and quartically convergent
algorithms for~$\pi$ can be expressed succinctly in terms of theta
functions. Inspection of these expressions suggests that some algorithms,
previously considered different, are actually equivalent,
in the sense that they give exactly the same sequence of approximations
to $\pi$ if performed using exact arithmetic.  For example, 
one of the Borweins' quadratically convergent
algorithms~\cite[Iteration 5.2 with $r=4$]{PAGM} 
is equivalent to
the Gauss-Legendre algorithm~\cite{rpb028,rpb034,Salamin76}, 
and it follows that one step of the Borweins'
quartically convergent 
algorithm~\cite[Iteration 5.3]{PAGM} 
is equivalent to two steps of the
Gauss-Legendre algorithm. These connections between superficially different
algorithms do not seem to have been noticed before.

In \S\ref{sec:prelim} we give some necessary definitions, discuss the 
\emph{arithmetic-geometric mean}, and consider its connection with
elliptic integrals and Jacobi theta functions. We also mention
the concept of \emph{order of convergence} of an algorithm.

A brief history of quadratically convergent algorithms for $\pi$
is given in \S\ref{sec:history}.

In \S\ref{sec:faster_than_linear} 
we consider some quadratically and quartically
convergent algorithms for $\pi$, including the Gauss-Legendre
algorithm and several algorithms due to the Borweins.
In \S\ref{sec:equivalence} we show that some of the algorithms
of \S\ref{sec:faster_than_linear}, although superficially different,
are actually equivalent when performed with exact arithmetic.

Chapter 5 of {\PAGM} considers some striking 
\emph{Ramanujan-Sato}
formul{\ae} for $1/\pi$ that
give very fast (though linearly convergent) algorithms for computing~$\pi$.
The first such formul{\ae} were given by Ramanujan~\cite{Ramanujan14}.
Later authors
include Takeshi Sato, the Borwein brothers, and the Chudnovsky brothers.
See~\cite{pi_next_gen,Baruah09,BBB89} for references.
In \S\ref{sec:Ramanujan-Sato} we briefly consider some Ramanujan-Sato formul{\ae}
and the corresponding algorithms for computing~$\pi$.

One of the ``icebergs'' alluded to above is the fast computation
of elementary functions to arbitrary precision.  The constant 
$\pi = 4\arctan(1)$
is of course just a special case (the tip of the iceberg).
In \S\ref{sec:elementary-fns} we outline how fast algorithms for computing
elementary (and some other) functions can be based on the
arithmetic-geometric mean iteration.

\section{Preliminaries: Means, Elliptic Integrals 
	and Theta Functions} \label{sec:prelim}

We define the \emph{order of convergence} of a sequence.
It will be sufficient to say that a sequence
$(x_n)_{n\in\nonnegint}$ 
\emph{converges linearly} to $L$ 
(or with \emph{order of convergence}~$1$) if
\[
0 < \mu_0 = 
\liminf_{n\to\infty} \frac{\left|x_{n+1}-L\right|}{\left|x_n-L\right|} 
\le
\limsup_{n\to\infty} \frac{\left|x_{n+1}-L\right|}{\left|x_n-L\right|} 
= \mu_1 < 1.
\]
If $\mu_0 = \mu_1$ then $\mu_0$ is called the \emph{rate of convergence}.

We say that a sequence $(x_n)_{n\in\nonnegint}$ 
\emph{converges to $L$ with order $p>1$} if the sequence
converges to $L$ and there exists
\[
p = \lim_{n\to\infty} \frac{\log\left|x_{n+1}-L\right|}
	{\log\left|x_{n}-L\right|} > 1.
\]
\emph{Quadratic, cubic} and \emph{quartic} convergence
are the cases $p = 2, 3, 4$ respectively.
For example, if $x_n = 2^n \exp(-3^n)$, then $(x_n)_{n\in\nonnegint}$ 
converges cubically to zero, because\linebreak
$\log|x_{n+1}|/\log|x_n| = (-3^{n+1}+O(n))/(-3^n+O(n)) \to 3$
as $n \to \infty$.

Roughly speaking, if a sequence converges linearly to $L$
with rate $\mu$, then the
number of correct decimal digits in the approximation to $L$
increases by about $\log_{10}(1/\mu)$ per term. 
For example, if 
\begin{equation}
x_n = 2\sqrt{3}\,\sum_{j=0}^{n}\frac{(-1)^j}{(2j+1)3^j}\,\raisecomma
					\label{eq:arctan_3}
\end{equation}
then $x_n$ converges linearly to $\pi$ with about 
$\log_{10}3 \approx 0.4771$ decimal digits per term.%
\footnote{The formula~\eqref{eq:arctan_3} is listed
in Bailey's compendium~\cite{Bailey-pi-formulae},
and is attributed to 
Madhava of Sangamagramma (c.1340--c.1425).
It follows from the Taylor series for $\arctan(1/\sqrt{3})$.}
If a sequence converges
to $L$ with order $p > 1$, then the number of correct digits is
approximately multiplied by $p$ for each additional term.
For example, Newton's method for computing square roots\footnote{Attributed
to Hero of Alexandria (c.$10$--$70$ A.D.), though also called the 
\emph{Babylonian method}.}
\[x_{n+1} := \frac{1}{2}\left(x_n + \frac{S}{x_n}\right)\]
converges quadratically to $L:=\sqrt{S}$, provided that
$x_0$ and $S$ are positive. In fact, it is easy to show that
\[x_{n+1}-L \approx \frac{1}{2L}(x_n-L)^2.\]

We now consider some well-known means.
The \emph{arithmetic mean} of $a, b \in \mathbb{R}$
is
\vspace*{-5pt}
\[\AM(a,b) := \frac{a+b}{2}\,\raisecomma\]
\vspace*{-5pt}
and the \emph{geometric mean} is
\begin{equation}
\GM(a,b) := \sqrt{ab}.		\label{eq:GM}
\end{equation}
Assuming that $a$ and $b$ are positive, we have the inequality
\[
\GM(a,b) \le \AM(a,b).\]

Initially we assume that $a$, $b$ are positive real numbers.
In \S\ref{sec:elementary-fns} we permit $a$, $b$ to be complex.
To resolve the ambiguity in the square root in~\eqref{eq:GM}
we assume that\linebreak
$\Re(GM(a,b)) \ge 0$, and $\Im(GM(a,b)) \ge 0$ if
$\Re(GM(a,b))=0$.

Given two positive reals $a_0, b_0$, we can iterate the arithmetic and
geometric means by defining, for $n \ge 0$,
\begin{align*}
a_{n+1} &= \AM(a_n, b_n)\\
b_{n+1} &= \GM(a_n, b_n).
\end{align*}

The sequences $(a_n)$ and $(b_n)$ converge 
quadratically to a common limit called the
\emph{arithmetic-geometric mean} (AGM) of $a_0$ and $b_0$.
We denote it
by $\AGM(a_0,b_0)$.  

Gauss~\cite{Gauss-Werke} and Legendre~\cite{Legendre}
solved the problem of expressing $\AGM(a,b)$ in terms of
known functions.  The answer may be written as
\begin{equation}
\frac{1}{\AGM(a,b)} = \frac{2}{\pi}\int_0^{\pi/2}
	\frac{d\theta}{\sqrt{a^2\cos^2\theta + b^2\sin^2\theta}}\,\raisedot
					\label{eq:recipAGM}
\end{equation}
The right-hand-side of~\eqref{eq:recipAGM} is the product of a constant
(whose precise value will be significant later) and a 
\emph{complete elliptic integral of the first kind}.
As usual, the {complete elliptic integral of the first kind} is defined by
\begin{equation*}
K(k)   := \int_0^{\pi/2} \frac{d\theta}{\sqrt{1-k^2\sin^2\theta}}
	= \int_0^1 \frac{dt}{\sqrt{(1-t^2)(1-k^2t^2)}}\,\raisecomma
\end{equation*}
and the \emph{complete elliptic integral of the second kind}
by
\begin{equation*}
E(k)   := \int_0^{\pi/2} \sqrt{1-k^2\sin^2\theta}\,d\theta
	= \int_0^1 \frac{\sqrt{1-k^2t^2}}{\sqrt{1-t^2}}\,dt.
\end{equation*}
The variable $k$ is called the \emph{modulus},
and $k' := \sqrt{1-k^2}$ is called the \emph{complementary modulus}.
It is customary to define 
\[K'(k) := K(\sqrt{1-k^2}) = K(k')\]
and
\[E'(k) := E(\sqrt{1-k^2}) = E(k'),\]
so in the context of elliptic integrals a prime ($'$)
does \emph{not} denote differentiation.
On the occasions when we need a derivative,
we use operator notation \[\deriv_k K(k) := dK(k)/dk.\]
We remark that {\PAGM} uses the 
``dot'' notation ${\dot{K}}(k) := dK(k)/dk$,
but this is potentially ambiguous and hard to see, so we prefer to avoid it.

The moduli
$k$ and $k'$ can in general be complex, 
but unless otherwise noted we assume that
they are real and in the interval $(0,1)$.

In terms of the Gaussian hypergeometric function
\[
F(a,b;c;z) := 1 + \frac{a\cdot b}{1!\cdot c}\,z + 
		\frac{a(a+1)\cdot b(b+1)}{2!\cdot c(c+1)}\,z^2 + \cdots,
\]
we have
\begin{equation}
K(k) = \frac{\pi}{2}\,F\left(\half,\half;1;k^2\right)
	\label{eq:K_hyper}
\end{equation}
and
\begin{equation}
E(k) = \frac{\pi}{2}\,F\left(-\half,\half;1;k^2\right).
	\label{eq:E_hyper}
\end{equation}
{From}~\eqref{eq:K_hyper} and \cite[17.3.21]{AS}, 
we also have\footnote{Here and elsewhere,
$\log$ denotes the natural logarithm.}
\begin{equation}
K'(k) = \frac{2}{\pi}\log\left(\frac{4}{k}\right)K(k) - f(k),
			\label{eq:Kprime_K}
\end{equation}
where $f(k) = k^2/4 + O(k^4)$ is analytic in the disk $|k| < 1$.

Substituting $(a,b) \mapsto (1,k')$ in~\eqref{eq:recipAGM}, and recalling that
$k^2 + (k')^2 = 1$, we have
\begin{equation}
\AGM(1,k') = \frac{\pi}{2K(k)}\,\raisedot  \label{eq:AGM-Kprime}
\end{equation}
Thus, if we start from $a_0=1$, $b_0=k'\in(0,1)$ and apply the
AGM iteration, $K(k)$ can be computed from
\begin{equation}
\lim_{n\to\infty}a_n = \frac{\pi}{2K(k)}\,\raisedot	\label{eq:lim_an}
\end{equation}
$E(k)$ can be computed via the AGM at the same time as $K(k)$,
using the well-known result~\cite[(b) on pg.~15]{PAGM}
\[
\frac{E(k)}{K(k)} = 1 - \frac{k^2}{2} -
		\sum_{n=0}^\infty 2^{n}\,(a_n-a_{n+1})^2.
\]

It follows from
\eqref{eq:K_hyper} and \eqref{eq:Kprime_K} that, for small $k$,
\begin{equation}
K'(k) = \left(1+O(k^2)\right)\log\left(\frac{4}{k}\right). \label{eq:logAGM}
\end{equation}
This will be relevant in 
\S\ref{sec:elementary-fns}.
A bound on the $O(k^2)$ term is given in 
\cite[Thm.~7.2]{PAGM}.

The Gauss-Legendre algorithm depends on 
\emph{Legendre's relation}: for $0 < k < 1$,
\[E(k)K'(k) + E'(k)K(k) - K(k)K'(k) = \frac{\pi}{2}\,\raisedot\]
For a proof, see \PAGM, Sec.~1.6.

A computationally important special case,
obtained by taking $k=k'=1/\sqrt{2}$, is
\begin{equation}
\left(2E\big(1/\sqrt{2}\big) - K\big(1/\sqrt{2}\big)\right) 
 K\big(1/\sqrt{2}\big) = 
  \frac{\pi}{2}\,\raisedot
	\label{eq:midpt-case}
\end{equation}
It can be shown~\cite[Thm.~1.7]{PAGM} that the two
factors in~\eqref{eq:midpt-case} are
\[
K(1/\sqrt{2}) =
\frac{\Gamma^2\left(\textstyle\frac{1}{4}\right)}{4\pi^{1/2}}
\;\;\text{ and }\;\;
2E(1/\sqrt{2}) - K(1/\sqrt{2}) =
\frac{\Gamma^2\left(\textstyle\frac{3}{4}\right)}{\pi^{1/2}}
\,\raisedot
\]

To estimate the order of convergence and to obtain error bounds,
we consider the parameterisation of the AGM in terms 
of \emph{Jacobi theta functions}.
We need the basic theta functions of one variable,
defined for $|q| < 1$ by
\[
\theta_2(q) := \sum_{n\in\mathbb{Z}}q^{(n+1/2)^2}, \;\;
\theta_3(q) := \sum_{n\in\mathbb{Z}}q^{n^2}, \;\;
\theta_4(q) := \sum_{n\in\mathbb{Z}}(-1)^n q^{n^2}.
\]
The theta functions satisfy many identities~\cite[\S21.3]{WW}. 
In particular, we use
the following addition formul{\ae}, due to Jacobi~\cite{Jacobi}.
They are proved in~\cite[\S2.1]{PAGM}.
\begin{align}
\theta_3^2(q) &= \theta_2^2(q^2) + \theta_3^2(q^2), \label{eq:J1}\\ 
\theta_3^4(q) &= \theta_2^4(q) + \theta_4^4(q). \label{eq:J2}	
\end{align}

It is not difficult to show that 
\[
\frac{\theta_3^2(q) + \theta_4^2(q)}{2} = \theta_3^2(q^2)
\;\;\text{and}\;\;
\sqrt{\theta_3^2(q)\theta_4^2(q)} = \theta_4^2(q^2).
\]
Thus, the AGM variables $(a_n, b_n)$ can
be parameterised by
$(\theta_3^2(q^{2^n}), \theta_4^2(q^{2^n}))$ if scaled suitably.
More precisely,
if $1 = a_0 > b_0 =  \theta_4^2(q)/\theta_3^2(q) > 0$,  
where $q \in (0,1)$, then the variables $a_n$, $b_n$
appearing in the AGM iteration satisfy
\begin{equation}
a_n = \frac{\theta_3^2(q^{2^n})}{\theta_3^2(q)}\,\raisecomma\;\;
b_n =  \frac{\theta_4^2(q^{2^n})}{\theta_3^2(q)}\,\raisedot
						\label{eq:AGM_a_b}
\end{equation}
It is useful to define auxiliary variables
$c_{n+1} := a_n - a_{n+1} = (a_n - b_n)/2$.  
Using the quotient for $a_n$ and the addition formula~\eqref{eq:J1},
we see that
\begin{equation}
c_n =  \frac{\theta_2^2(q^{2^n})}{\theta_3^2(q)}
						\label{eq:AGM_c}
\end{equation}
holds for $n \ge 1$. We could use \eqref{eq:AGM_c} to define $c_0$,
but this will not be necessary.\footnote{Salamin~\cite{Salamin76}
defines $c_n$ using the relation $c_n^2 = a_n^2-b_n^2$.
This has the advantage that $c_0$ is defined naturally, and 
for $n > 0$ it is equivalent to our definition. However, it is
computationally more expensive to compute
$(a_n^2-b_n^2)^{1/2}$ than $a_n-a_{n+1}$.
}

We can write $q$ (which is called the \emph{nome}) explicitly,
in fact
\begin{equation}
q = \exp(-\pi K'(k)/K(k)).			\label{eq:nome}
\end{equation} 
This is due to Gauss/Jacobi; for a proof see~\cite[Thm.~2.3]{PAGM}.
In the important special case $k = k' = 1/\sqrt{2}$, we have $K'=K$ and
$q = e^{-\pi} = 0.0432139\ldots$

Because the AGM iteration converges quadratically, it offers the prospect
of quadratically convergent algorithms for approximating $\pi$ and,
more generally, all the elementary functions. This is the topic of
\S\ref{sec:faster_than_linear} and \S\ref{sec:elementary-fns} below. 
First we make some comments on the history of quadratically convergent
algorithms for~$\pi$.

\section{Historical Remarks}		\label{sec:history}

An algorithm for computing $\log(4/k)$, using~\eqref{eq:AGM-Kprime},
\eqref{eq:logAGM} and the AGM,
assuming that we know $\pi$ to sufficient accuracy,
was given by Salamin~\cite[pg.~71]{Salamin-HAKMEM} in 1972.
On the same page Salamin gives an algorithm for computing $\pi$,
taking $k = 4/e^n$ in~\eqref{eq:logAGM}.
With his choice $\pi \approx 2n\,\AGM(1,k)$.
However, this assumes that we know $e$, so it is not a 
``standalone'' algorithm for $\pi$ via the AGM.
Similarly, if we take $k = 4/2^n$ in~\eqref{eq:logAGM}, 
we obtain an algorithm
for computing $\pi\log 2$ (and hence $\pi$, if we know $\log 2$).

In 1975, Salamin~\cite{Salamin76} and (independently)
the present author~\cite{rpb028,rpb034}
discovered a quadratically convergent
algorithm for computing $\pi$ via the AGM \emph{without} needing to know
$e$ or $\log 2$ to high precision. 
It is known as the ``Gauss-Legendre'' algorithm (after the discoverers of
the key identities~\cite{Gauss-pi,Legendre}) or the ``Brent-Salamin''
algorithm (after the 20th century discoverers~\cite{rpb252}), 
and is about twice as fast as the earlier algorithms which assume
a knowledge of $e$ or $\log 2$.
We abbreviate the name to \emph{Algorithm GL}.
Bailey and Borwein, in 
\emph{Pi: The Next Generation}~\cite[Synopsis of paper 1]{pi_next_gen},
say ``This remarkable co-discovery arguably launched the
modern computer era of the computation of $\pi$''.%
\footnote{In~\cite[\S10]{Borwein-Berggren}, Jon Borwein says
``It [Algorithm GL]
is based on the arithmetic-geometric mean iteration (AGM) and some
other ideas due to Gauss and Legendre around 1800, although neither Gauss,
nor many after him, ever directly saw the connection to effectively
computing~$\pi$''.}

In 1984, Jon and Peter Borwein~\cite{Borwein84a}
(see also \cite[Alg.~2.1]{PAGM})
discovered another quadratically
convergent algorithm
for computing $\pi$, with convergence about as fast as Algorithm~GL.
We call this the (first) \emph{Borwein-Borwein algorithm},
or {\emph{\hbox{Algorithm} BB1}}.
Yet another quadratically convergent algorithm, which we call the
(second) \emph{Borwein-Borwein algorithm} and abbreviate as
\emph{Algorithm BB2}, 
dates from 1986~--
see~\cite{Borwein86} and \cite[Iteration 5.1]{PAGM}.
Although Algorithm BB2 appears different from Algorithm GL, we show
in \S\ref{sec:equivalence} that the two algorithms are in fact
equivalent, in the sense of producing the same sequence of approximations
to $\pi$. This surprising fact does not seem to have been noticed before.

\section[Some Superlinearly Convergent Algorithms for Pi]%
{Some Superlinearly Convergent Algorithms for $\pi$}
	\label{sec:faster_than_linear}

In this section we describe the Gauss-Legendre algorithm (GL)
and two quadratically convergent algorithms (BB1 and BB2)
due to Jon and Peter Borwein.
We also describe a $4$-th order algorithm (BB4) due to
the Borweins.

Using Legendre's relation and the formul{\ae} that we have given for $E$ and
$K$ in terms of the AGM iteration, it is not difficult to derive
Algorithm~GL. We present it in pseudo-code using the same style as
the algorithms in~\cite{rpb226}.
\vspace*{10pt}

\pagebreak[3]
\begin{samepage}
\noindent\textbf{Algorithm GL}\\
\textbf{Input}: The number of iterations $n_{max}$.\\
\textbf{Output}: A sequence of $n_{max}$ intervals containing $\pi$.
\begin{align*}
&a_0 := 1;\; b_0 := 1/\sqrt{2};\; s_0 := \textstyle\frac{1}{4}.\\
&\textbf{for } n \text{ from } 0 \text{ to } n_{max}-1 \text{ do }\\
&\hspace*{2em}a_{n+1} := (a_n+b_n)/2;\\
&\hspace*{2em}c_{n+1} := a_n - a_{n+1};\\	
&\hspace*{2em}\textbf{output } (a_{n+1}^2/s_n,\; a_{n}^2/s_n).\\
&\hspace*{2em}\textbf{if } n < n_{max}-1 \textbf{ then}\\[-5pt]
&\hspace*{4em}b_{n+1} := \sqrt{a_n b_n};\\
&\hspace*{4em}s_{n+1} := s_n - 2^n\,c_{n+1}^2.
\end{align*}
\end{samepage}
\pagebreak[3]

\noindent\textbf{Remarks}
\begin{enumerate}
\item
Subscripts on variables such as $a_n, b_n$ are given for
expository purposes. In an efficient implementation only a constant number of
real variables are needed, because $a_{n+1}$ can overwrite $a_n$
(after saving $a_n$ in a temporary variable for use in the computation
of $b_{n+1}$ and $c_{n+1}$), and similarly for $b_n$, $c_n$ and $s_n$.
\item
The purpose of the final ``if $\ldots$ then'' is simply to avoid
unnecessary computations after the final output.
Similar comments apply to the other algorithms given below.
\item
Salamin~\cite{Salamin76} notes the identity
$4a_{n+1}c_{n+1} = c_n^2$ which can be used to compute $c_{n+1}$ 
without the numerical cancellation that occurs when using
the definition $c_{n+1} = a_n - a_{n+1}$. 
However, this refinement costs time and is unnecessary, because the
terms $2^n c_{n+1}^2$ diminish rapidly and make only a minor
contribution to the overall error caused by using finite-precision
real arithmetic. To obtain an accurate result it is sufficient
to use $O(\log n_{max})$ guard digits.
\end{enumerate}

\noindent
Neglecting the effect of rounding errors,
Algorithm GL gives a sequence of lower and upper bounds on $\pi$:
\[\frac{a_{n+1}^2}{s_n} < \pi < \frac{a_{n}^2}{s_{n}}\,\raisecomma\]
and both bounds converge quadratically to $\pi$.
The lower bound is more accurate, so the algorithm is often stated with
just the lower bound $a_{n+1}^2/s_n$
(we call this variant \emph{Algorithm GL1}). 
Table \ref{tab:a} shows the approximations
to $\pi$ given by the first few iterations.  Correct digits are shown in
bold.  The quadratic convergence is evident.

\begin{table}[ht]
\vspace*{-0pt}
\centering              
\begin{tabular}{cccc}
$n$ & lower bound $a_{n+1}^2/s_n$ & & upper bound $a_n^2/s_n$ \\[2pt]\hline 
$0\;$ & 2.914213562373095048801689 & $< \pi <$ & 
	4.000000000000000000000000\\
$1\;$ & {\correct 3.14}0579250522168248311331 & $< \pi <$ &
	{\correct 3.1}87672642712108627201930\\
$2\;$ & {\correct 3.1415926}46213542282149344 & $< \pi <$ &
	{\correct 3.141}680293297653293918070\\
$3\;$ & {\correct 3.141592653589793238}279513 & $< \pi <$ &
	{\correct 3.141592653}895446496002915\\
$4\;$ & {\correct 3.141592653589793238462643} & $< \pi <$ &
	{\correct 3.14159265358979323846}6361\\
\hline
\end{tabular}
\vspace*{5pt}
\caption{Convergence of Algorithm GL}
\label{tab:a}
\end{table}

Recall that in Algorithm GL we have
$a_0 = 1$, $b_0 = 1/\sqrt{2}$, $s_0 = \frac{1}{4}$
and, for $n \ge 0$,
\[
a_{n+1} = \frac{a_n+b_n}{2}\,\raisecomma\;\; b_{n+1} = \sqrt{a_n b_n},\;\; 
c_{n+1} = a_n - a_{n+1},\;\;
  s_{n+1} = s_n - 2^n\,c_{n+1}^2\,.
\]

\pagebreak[3]
\noindent
Take $q = e^{-\pi}$, and 
write 
\begin{align}
a_\infty :=& \lim_{n\to\infty}a_n = \theta_3^{-2}(q)
	= 2\pi^{3/2}/\Gamma^2({\textstyle\frac14}) \approx 0.8472,
							\label{eq:ainf}\\
s_\infty :=& \lim_{n\to\infty}s_n
	= \theta_3^{-4}(q)/\pi
	= 4\pi^2/\Gamma^4({\textstyle\frac14}) \approx 0.2285\,.
							\label{eq:sinf}
\end{align}
Since $c_n = \theta_2^2(q^{2^n})/\theta_3^2(q)$,
we have	
\begin{equation}
s_n - s_\infty = \theta_3^{-4}(q) \sum_{m=n}^\infty 2^{m}
	\theta_2^4(q^{2^{m+1}})\,.
		\label{eq:s_n_s2}
\end{equation}
Write $a_n/a_\infty = 1+\delta_n$ and $s_n/s_\infty = 1+\varepsilon_n$.
Then 
\[\delta_n = \theta_3^2(q^{2^n}) - 1 \sim 4q^{2^n}\;\;\text{as}\;\;
	n \to \infty,\]
and \eqref{eq:sinf}~--~\eqref{eq:s_n_s2} give
\[
\varepsilon_n = \pi\sum_{m=n}^\infty
        2^m\,\theta_2^4(q^{2^{m+1}})
	\sim 2^{n+4}\pi q^{2^{n+1}}.
\]

Writing 
\[\frac{a_n^2/a_\infty^2}{s_n/s_\infty}
 = \frac{a_n^2}{\pi s_n}
 = \frac{(1+\delta_n)^2}{1+\varepsilon_n}\,\raisecomma\]
it is straightforward to obtain an upper bound on $\pi$:
\begin{equation}
0 < a_n^2/s_n - \pi < U(n) := 8\pi q^{2^n}.
					\label{eq:GS_upper1}
\end{equation}
Convergence is quadratic: if $e_n := a_n^2/s_n - \pi$, then
\[	
\lim_{n\to\infty} {e_{n+1}}/{e_n^2} = {\textstyle\frac{1}{8\pi}}\,\raisedot
\] 
Replacing $a_n$ by $a_{n+1}$ and $\delta_n$ by $\delta_{n+1}$,
we obtain a lower bound on $\pi$:
\begin{equation}
0 < \pi - \frac{a_{n+1}^2}{s_n} < L(n) :=
  (2^{n+4}\pi^2-8\pi) q^{2^{n+1}}.
					\label{eq:GS_lower1}
\end{equation}
{\PAGM} [(2.5.7) on page 48]
gives a slightly weaker lower bound which, via \eqref{eq:ainf},
may be written as 
\begin{equation}
\pi - \frac{a_{n+1}^2}{s_n} \le 
  \frac{2^{n+4}\pi^2 q^{2^{n+1}}}{a_\infty^2}\,\raisedot
					\label{eq:GS_lower2}
\end{equation}
Since $a_\infty^2 < 1$,
the bound~\eqref{eq:GS_lower2} is weaker than the
bound~\eqref{eq:GS_lower1}.
In~\eqref{eq:GS_lower1},
the factor\linebreak 
$(2^{n+4}\pi^2 - 8\pi)$
is the best possible, since an
expansion of $a_{n+1}^2/s_n$ in powers of $q$ gives
$\pi - {a_{n+1}^2}/{s_n} = (2^{n+4}\pi^2-8\pi) q^{2^{n+1}}
	\! - \,O(2^nq^{2^{n+2}})$, with the minus sign before the ``$O$'' term
informally indicating the sign of the remainder.

\begin{table}[ht]
\centering              
\begin{tabular}{c|c|c|c|c}
$\;n\;$ & $a_n^2/s_n - \pi$ & $\pi - a_{n+1}^2/s_n$
   & $\displaystyle\frac{a_n^2/s_n - \pi}{U(n)}$ 
   & $\displaystyle\frac{\pi - a_{n+1}^2/s_n}{L(n)}$\\[5pt]\hline
0 & 8.58e-1\;\; & 2.27e-1\;\; & $\;$0.790369040$\;$ & $\;$0.916996189$\;$ \\
1 & 4.61e-2\;\; & 1.01e-3\;\; & 0.981804947 & 0.999656206 \\
2 & 8.76e-5\;\; & 7.38e-9\;\; & 0.999922813 & 0.999999998 \\
3 & 3.06e-10\; & 1.83e-19\; & 0.999999999 & 1.000000000 \\
4 & 3.72e-21\; & 5.47e-41\; & 1.000000000 & 1.000000000 \\
5 & 5.50e-43\; & 2.41e-84\; & 1.000000000 & 1.000000000 \\
6 & 1.20e-86\; & 2.31e-171 & 1.000000000 & 1.000000000 \\
7 & 5.76e-174 & 1.06e-345 & 1.000000000 & 1.000000000 \\
8 & 1.32e-348 & 1.11e-694 & 1.000000000 & 1.000000000 \\
\hline
\end{tabular}
\vspace*{5pt}
\caption{Numerical values of upper and lower bounds for Algorithm~GL}
\label{tab:b}
\end{table}
In Table~\ref{tab:b},
$U(n) := 8\pi\exp(-2^n\pi)$ and
$L(n) := (2^{n+4}\pi^2 - 8\pi)\exp(-2^{n+1}\pi)$ are the bounds
given in~\eqref{eq:GS_upper1}--\eqref{eq:GS_lower1}.
It can be seen that the bounds are
very accurate for $n > 1$, as expected from our analysis.

Recall that Algorithm GL gives
approximations $a_n^2/s_n$ and $a_{n+1}^2/s_n$ 
to \hbox{$\pi = a_\infty^2/s_\infty$}.
Using the expressions for $a_n$ and $s_n$ in terms of theta functions, we
see that
\begin{equation}	
\pi = \frac{a_n^2\;\theta_3^{-4}(q^{2^n})}
	{s_n - \theta_3^{-4}(q)
	\sum_{m=n}^\infty 2^{m}\,
        {\theta_2^4(q^{2^{m+1}})}} \,\raisecomma
							\label{eq:pi_family}
\end{equation}
[or similarly with the numerator replaced by
$a_{n+1}^2\theta_3^{-4}(q^{2^{n+1}})$].
The expression~\eqref{eq:pi_family} for $\pi$ is essentially of the form 
\[\pi = 
	\frac{a_n^2 - O(q^{2^n})}{s_n - O(2^nq^{2^{n+1}})}
\;\;\left[\text{or}\;\;
	\frac{a_{n+1}^2 - O(q^{2^{n+1}})}{s_n - O(2^nq^{2^{n+1}})}
\right].
\]
This shows precisely how Algorithm~GL approximates $\pi$
and why it provides upper {[}or lower{]} bounds.

In \PAGM, Jon and Peter Borwein present a quadratically convergent
algorithm for $\pi$, based on the AGM,
but different from Algorithm~GL.  
It is Algorithm 2.1 in Chapter 2,
and was first published in~\cite{Borwein84a}. 
We call it \emph{Algorithm~BB1}.

Instead of using Legendre's relation, Algorithm BB1 uses
the identity
\[ 
K(k) \;\deriv_k K(k)\big|_{k=1/\sqrt{2}}\; = \frac{\pi}{\sqrt{2}}\,\raisecomma
\]
where $\deriv_k$ denotes differentiation with respect to $k$.

Using the connection between $K(k')$ and the AGM, the Borweins
\cite[(2.4.7)]{PAGM} prove that
\[	
\pi = \,2^{3/2}\left.\frac {(\AGM(1,k'))^3}
		   {\deriv_k\, \AGM(1,k')}\right|_{k=1/\sqrt{2}}\;.
\]
An algorithm for approximating the derivative in this formula
can be obtained by differentiating the AGM iteration
symbolically. Details are given in \cite{PAGM}. 

We now present Algorithm~BB1.
Note that the algorithm given in \cite{PAGM}
defines the upper bound $\pihigh_n := \pihigh_{n-1}(x_n+1)/(y_n+1)$ 
and omits the lower bound ${\pilow}_n$,
but ${\pilow}_n$ can be obtained from~\cite[ex.~2.5.11]{PAGM}.
We present a version that computes upper $(\pihigh_n)$
and lower $(\pilow_n)$ bounds for comparison with Algorithm~GL.
\vspace*{10pt}

\pagebreak[3]			
\begin{samepage}
\noindent\textbf{Algorithm BB1}\\ 
\textbf{Input}: The number of iterations $n_{max}$.\\
\textbf{Output}: A sequence of $n_{max}$ intervals containing $\pi$.
\begin{align*}
&x_0 := \sqrt{2};\\
&\textbf{output } ({\pilow}_0 := x_0,\; {\pihigh}_0 := x_0+2).\\
&y_1 := {x_0}^{1/2}; 
 \;x_1 := {\textstyle\frac12}(x_0^{1/2}+x_0^{-1/2});\\
&\textbf{for } n \text{ from } 1 \text{ to } n_{max}-1 \text{ do }\\
&\hspace*{2em}{\pilow}_n := \frac{2\,{\pihigh}_{n-1}}{y_n+1};\;
 {\pihigh}_n := {\pilow}_n\left(\frac{x_n+1}{2}\right);\\
&\hspace*{2em}\textbf{output } ({\pilow}_n,\; {\pihigh}_n);\\
&\hspace*{2em}\textbf{if } n < n_{max}-1 \textbf{ then} \\[-5pt]
&\hspace{4em}x_{n+1} := {\textstyle\frac12}(x_n^{1/2} + x_n^{-1/2});\;\;
y_{n+1} := \frac{y_n\,x_n^{1/2} + x_n^{-1/2}}{y_n+1}\,\raisedot
\end{align*}
\end{samepage}
It may be shown that ${\pihigh}_n$ decreases monotonically to the limit $\pi$,
and ${\pilow}_n$ increases monotonically
to $\pi$.	
Moreover, ${\pihigh}_n - {\pilow}_n$ decreases quadratically to zero.
This is illustrated in Table~\ref{tab:c}.

It is not immediately obvious that Algorithm BB1 depends on the AGM.
However, the AGM is present in \emph{Legendre form}:
if $a_0 := 1$, $b_0 := k' = 1/\sqrt{2}$,
and we perform $n$ steps of the AGM iteration
to define $a_n, b_n$, then $x_n = a_n/b_n$
and, for $n \ge 1$,  $y_n = \deriv_k b_n/\deriv_k a_n$.

\begin{table}[ht]
\centering              
\begin{tabular}{cccc}
$n$ 	& ${\pilow}_n$ 	&	& ${\pihigh}_n$ \\[3pt] \hline
0	& {\correct }1.414213562373095048801689 &$< \pi <$ 
	& {\correct 3}.414213562373095048801689 \\
1	& {\correct 3.1}19132528827772757303373 &$< \pi <$ 
	& {\correct 3.14}2606753941622600790720 \\
2	& {\correct 3.1415}48837729436193482357 &$< \pi <$ 
	& {\correct 3.1415926}60966044230497752 \\
3	& {\correct 3.141592653}436966609787790 &$< \pi <$
	& {\correct 3.141592653589793238}645774 \\
4	& {\correct 3.14159265358979323846}0785 &$< \pi <$
	& {\correct 3.141592653589793238462643} \\
\hline
\end{tabular}
\vspace*{5pt}
\caption{Convergence of Algorithm BB1}
\label{tab:c}
\end{table}

Comparing Tables~\ref{tab:a} and~\ref{tab:c}, we see that
Algorithm BB1 gives better upper bounds, but worse lower bounds,
than Algorithm GL,
for the same value of $n$ (i.e.\ same number of square roots).

As for Algorithm GL, we can express the error after $n$
iterations of Algorithm BB1 using theta functions,
and deduce the asymptotic behaviour of the error.

Consider the AGM iteration with $a_0 = 1, b_0 = k' = (1-k^2)^{1/2}$. Then 
$a_n$ and $b_n$ are functions of $k$.
In {\PAGM} 
it is shown that, for $n \ge 1$,
\begin{equation}
\pihigh_{n-1} = \left(2^{3/2}b_n^2a_n/\deriv_k a_n\right)|_{k=1/\sqrt{2}}\,.
	\label{eq:PAGM_pg47}
\end{equation}
Now $a_n$ and $b_n$ are given by~\eqref{eq:AGM_a_b} with $q = e^{-\pi}$. 
We differentiate $a_n$
with respect to~$k$, 
where $k = (1-b_0^2)^{1/2} = \theta_2^2(q)/\theta_3^2(q)$.
This gives
\begin{equation}
\deriv_k a_n = \deriv_q\!\left(\frac{\theta_3^2(q^{2^{n}})}
	{\theta_3^2(q)}\right)
	\!\Big{/}\!
	\deriv_q\!\left(\frac{\theta_2^2(q)}{\theta_3^2(q)}\right)
	\Bigg{|}_{q=e^{-\pi}}
	\,.					\label{eq:Dkan}
\end{equation}
We remark that~\eqref{eq:Dkan}
gives $\deriv_k a_0 = 0$, as expected since $a_0$ is 
independent of $k$.

Thanks to the analyticity of the theta functions in $|q| < 1$, 
there is no difficulty in showing that\footnote{Similarly, where we
exchange the order of taking derivatives and limits elsewhere
in this section, it is easy to justify.} 
\[\lim_{n\to\infty}\deriv_k a_n = \deriv_k \lim_{n\to\infty}a_n\,.\]
We denote the common value by $\deriv_k a_\infty$.
Taking the limit in \eqref{eq:PAGM_pg47},
we obtain (as also follows
from~\cite[(2.4.7)]{PAGM}): 
\begin{equation}
\deriv_k a_\infty = \frac{2^{3/2}a_\infty^3}{\pi} = 0.547486\ldots
				\label{eq:deriv_k_a}
\end{equation}
Now 
$a_n - a_\infty = \displaystyle{\sum_{m=n+1}^\infty c_m},$
and differentiating both sides with respect to $k$ gives
\begin{equation}	
\deriv_k a_n - \deriv_k a_\infty = 
\sum_{m=n+1}^\infty \deriv_q\!\left(\frac{\theta_2^2(q^{2^{m}})}
        {\theta_3^2(q)}\right)
        \!\Big{/}\!
        \deriv_q\!\left(\frac{\theta_2^2(q)}{\theta_3^2(q)}\right)
        \Bigg{|}_{q=e^{-\pi}}
        \,.					\label{eq:Dkainf}
\end{equation}
We remark that \eqref{eq:Dkainf} is analogous to \eqref{eq:s_n_s2}, which
we used in the analysis of Algorithm~GL.
Using \eqref{eq:PAGM_pg47}~-- \eqref{eq:Dkainf},
we obtain an upper bound on $\pi$ (for $n \ge 1$, $q = e^{-\pi}$)
\begin{equation}
0 < \pihigh_n - \pi < 2^{n+4}\pi^2 q^{2^{n+1}}.
				\label{eq:BB1_upper}
\end{equation}
A slightly weaker bound than~\eqref{eq:BB1_upper} is
proved in~\cite[\S2.5]{PAGM}.

\pagebreak[3]
Similarly, we can obtain a lower bound on $\pi$:
\begin{equation}
0 < \pi - \pilow_n < 4\pi q^{2^n}.
				\label{eq:BB1_lower}
\end{equation}
We omit detailed proofs of \eqref{eq:BB1_upper} and \eqref{eq:BB1_lower};
they involve straightforward but tedious expansions of power
series in~$q$.
\hbox{Experimental} evidence is provided in Table~\ref{tab:d}.

\begin{table}[ht]
\hspace*{-20pt}
\centering		
\begin{tabular}{c||c|c||c|c}
$n$ & $\pihigh_n-\pi$ 
	& $\displaystyle\frac{\pihigh_n-\pi}{2^{n+4}\pi^2q^{2^{n+1}}}$ 
	& $\pi - \pilow_n$ 
	& $\displaystyle\frac{\pi-\pilow_n}{4\pi q^{2^n}}$\\[5pt]
\hline
1 & 1.01e-3\;\; & 0.9896487063 & 2.25e-2 & 0.9570949132 \\
2 & 7.38e-9\;\; & 0.9948470082 & 4.38e-5 & 0.9998316841 \\
3 & 1.83e-19\; & 0.9974691480 & 1.53e-10 & 0.9999999988 \\
4 & 5.47e-41\; & 0.9987456847 & 1.86e-21 & 1.0000000000 \\
5 & 2.41e-84\; & 0.9993755837 & 2.75e-43 & 1.0000000000 \\
6 & 2.31e-171\; & 0.9996884727 & 6.01e-87 & 1.0000000000 \\
7 & 1.06e-345\; & 0.9998444059 & 2.88e-174\; & 1.0000000000 \\
8\; & 1.11e-694\; &0.9999222453 & 6.59e-349\; & 1.0000000000 \\
\hline
\end{tabular}
\vspace*{5pt}
\caption{Numerical values of upper and lower bounds for Algorithm BB1}
\label{tab:d}
\end{table}

Table~\ref{tab:d} gives numerical values of the
approximation errors $\pihigh_n - \pi$ and
$\pi - \pilow_n$, and the ratio of these values to the bounds
\eqref{eq:BB1_upper} and \eqref{eq:BB1_lower} respectively.
It can be seen that the bounds are very accurate (as expected from the
expressions for the errors in terms of theta functions
and the rapid convergence of the series for the theta functions).
The upper bound overestimates the error by a factor
of $1 + O(2^{-n})$.  A computation shows that we can not replace the 
bound by the function
$L(n)$ defined in~\eqref{eq:GS_lower1}, although a similar bound appears
to be valid if the constant $8\pi$ in~\eqref{eq:GS_lower1} is replaced
by a slightly smaller constant, e.g.~$7\pi$.

The bounds \eqref{eq:BB1_upper}--\eqref{eq:BB1_lower}
can be compared with the
{lower bound} $(2^{n+4}\pi^2-8\pi) q^{2^{n+1}}$ and
{upper bound} $8\pi q^{2^n}$ for Algorithm GL.
The upper bound is better for Algorithm BB1, but the lower bound is better
for Algorithm GL.  This confirms the observation above regarding the
comparison of Tables \ref{tab:a} and \ref{tab:c}.

Since it will be needed in \S\ref{sec:equivalence},
we state another quadratic algorithm, 
\emph{Algorithm BB2}, different from Algorithm BB1 but also due to Jon
and Peter Borwein (iteration $5.2$ on page 170 of \cite{PAGM} with
the parameter $r=4$).

\vspace*{10pt}
\pagebreak[3]
\begin{samepage}
\noindent\textbf{Algorithm BB2}\\
\textbf{Input}: The number of iterations $n_{max}$.\\
\textbf{Output}: A sequence of $n_{max}$ approximations to $\pi$.
\begin{align*}
&\alpha_0 := 6-4\sqrt{2};\;\; k_0 := 3-2\sqrt{2};\\
&\textbf{for } n \text{ from } 0 \text{ to } n_{max}-1 \text{ do}\\
&\hspace*{2em}\textbf{output } \widehat{\pi}_n := 1/\alpha_{n}\,;\\
&\hspace*{2em}\textbf{if } n < n_{max}-1 \textbf{ then }\\
&\hspace*{4em} k_n' := \sqrt{1-k_n^2};\;\; k_{n+1} := \frac{1-k_n'}{1+k_n'}\,;\\
&\hspace*{4em}
 \alpha_{n+1} := (1+k_{n+1})^2\alpha_n - 2^{n+2}k_{n+1}\,.
\end{align*}
\end{samepage}
\pagebreak[3]
\noindent
In Algorithm BB2, we
have $\widehat{\pi}_n \to \pi$ quadratically
\cite[pg.~170]{PAGM}. 
We remark that it would be clearer to increase (by one) the subscripts on the
variables in Algorithm BB2, so as to correspond to the usage in
Algorithm GL, which implicitly has {$k_0' = b_0/a_0 = 1/\sqrt{2}$}
and $k_1 = (1-k_0')/(1+k_0') = 3 - 2\sqrt{2}$, 
but we have kept the notation used in~\cite{PAGM}.

\pagebreak[3]
The Borwein brothers did not stop at quadratic (second-order) algorithms 
for~$\pi$. In 
Chapter~5 of \PAGM\ they gave
algorithms of orders 3, 4, 5 and 7.
Of course, these algorithms are not necessarily faster than the quadratic
algorithms, because 
we must take into account the amount of work per iteration.
For a fair comparison, we can use
Ostrowski's \emph{efficiency index}~\cite[\S3.11]{Ostrowski},
\hbox{defined} as
$\log(p)/W$, where $p > 1$ is the order of convergence and
$W$ is the work per iteration.
A justification of this measure of efficiency is given in~\cite{rpb014}.
Consider a simple example~-- if we combine three iterations of
Algorithm BB2 into one iteration of a new algorithm, then we
obtain an algorithm of order $8$, but with three times as much work
per iteration.  The efficiency index is the same in both cases,
as it should be.

We refer to \cite[Chapter~5]{PAGM} for the Borweins'
cubic, quintic and higher-order
algorithms, and consider only their
quartic algorithm, which we call \emph{Algorithm BB4}.
It is a specialisation to the case $r=4$ of the
slightly more general algorithm given in
\cite[iteration 5.3, pg.~170]{PAGM}.
The same special case is given in~\cite[Algorithm~1]{BBB89}
and has been used in extensive calculations of $\pi$, see for example
\cite{Bailey88,Kanada88}.
We have changed notation slightly ($a_n \mapsto z_n$) to avoid
conflict with the notation used in Algorithm~GL.

\vspace*{10pt}
\pagebreak[3]			
\begin{samepage}
\noindent\textbf{Algorithm BB4}\\
\textbf{Input}: The number of iterations $n_{max}$.\\
\textbf{Output}: A sequence of $n_{max}$ approximations to $\pi$.
\begin{align*}
&y_0 := \sqrt{2}-1;\;\; z_0 := 2y_0^2;\\
&\textbf{for } n \text{ from } 0 \text{ to } n_{max}-1 \text{ do}\\
&\hspace*{2em}\textbf{output } \pi_n := 1/z_n\,;\\
&\hspace*{2em}\textbf{if } n < n_{max}-1 \textbf{ then }\\
&\hspace{4em}y_{n+1} := \frac{1-(1-y_n^4)^{1/4}}
		  {1+(1-y_n^4)^{1/4}}\;;\\
&\hspace{4em}z_{n+1} := 
		z_n(1+y_{n+1})^4 - 2^{2n+3}y_{n+1}(1+y_{n+1}+y_{n+1}^2).
\end{align*}
\end{samepage}

\pagebreak[3]
\noindent
In Algorithm BB4, $\pi_n$ converges quartically to $\pi$.
A sharp error bound is
\begin{equation}
0 < \pi - \pi_n < \pi^2\,4^{n+2}\exp(-2\pi\, 4^n).
				\label{eq:BB4_bound}
\end{equation}
This improves by a factor of two on the error bound given
in \cite[top of pg.~171]{PAGM}.
We defer the proof until \S\ref{sec:equivalence}.

Table \ref{tab:e} shows the error $\pi-\pi_n$ after $n$ iterations of the
Borwein quartic algorithm, and the ratio 
of the error $\pi-\pi_n$ to the upper bound~\eqref{eq:BB4_bound}.

\begin{table}[ht]
\centering              
\begin{tabular}{c|l|c}
$\;n\;$ & \hspace*{3em}$\pi-\pi_n$ & 
 $\displaystyle\frac{\pi-\pi_n}{\text{bound \eqref{eq:BB4_bound}}}$\\[5pt]
\hline
0 & \, 2.273790912e-1 & \, 0.7710517124 \\
1 & \, 7.376250956e-9 & \, 0.9602112619 \\
2 & \, 5.472109145e-41 & \, 0.9900528160 \\
3 & \, 2.308580715e-171 & \, 0.9975132040 \\
4 & \, 1.110954934e-694 & \, 0.9993783010 \\
5 & \, 9.244416653e-2790 & \, 0.9998445753 \\
6 & \, 6.913088685e-11172 & \, 0.9999611438 \\
7 & \, 3.376546688e-44702 & \, 0.9999902860 \\
8 & \, 3.002256862e-178825 & \, 0.9999975715 \\
\hline
\end{tabular}
\vspace*{5pt}
\caption{Approximation error in Algorithm BB4}
\label{tab:e}
\end{table}

At this point the reader may well ask ``which of Algorithms GL, BB1, BB2
and BB4 is the fastest?''. The answer seems to depend on implementation
details.  All four algorithms involve the same number of square roots to
obtain comparable accuracy (counting a fourth root in Algorithm BB4 as
equivalent to two square roots, which is not necessarily correct%
\footnote{For example, one might compute $x^{1/4}$ using two
\emph{inverse} square roots, i.e. $(x^{-1/2})^{-1/2}$, which is possibly 
faster than two square roots, i.e. $(x^{1/2})^{1/2}$, 
see~\cite[\S4.2.3]{rpb226}.}).
Algorithm~GL has the advantage that high-precision divisions are only
required when generating the output (so the early divisions can be skipped
if intermediate output is not required).
The other three algorithms require at least one division per iteration.
Borwein, Borwein and Bailey~\cite[pg.~202]{BBB89} say 
``[Algorithm BB4] is arguably the most efficient algorithm
currently known for the extended precision calculation of $\pi$'',
and the times given in Bailey's paper~\cite[pg.~289]{Bailey88} confirm this
(28 hours for Algorithm BB4 versus 40 hours for Algorithm BB1).
However, Kanada~\cite{Kanada88}, who extended Bailey's computation,
reached the opposite conclusion. His computation took 5 hours 57 minutes
with Algorithm~GL, and 7 hours 30 minutes with Algorithm~BB4 (which was used
for verification).
\pagebreak[3]

\section[Equivalence of Some Algorithms for Pi]%
{Equivalence of Some Algorithms for $\pi$}
	\label{sec:equivalence}

In the following, 
\emph{doubling} an algorithm $A$ means to construct an algorithm $A^2$ 
that outputs
$(x_0,x_2,x_4,\ldots)$ if algorithm $A$ outputs
$(x_0,x_1,x_2,\ldots)$.
Replacing $n$ by $2n$ in~\eqref{eq:GS_lower1} and retaining only the most
significant term, we see that an error bound for Algorithm GL1 doubled is
\[0 < \pi - a_{2n+1}^2/s_{2n} < \pi^2\,4^{n+2}\exp(-2\pi\,4^n).\]
It is suggestive that the right-hand side 
is the same as in the error bound~\eqref{eq:BB4_bound}
for the Borwein quartic algorithm after $n$ iterations.

On closer inspection we find
that the two algorithms (GL1 doubled
and BB4) are \emph{equivalent}, in the sense that they give
\emph{exactly} the same sequence of approximations to $\pi$.
Symbolically,
\begin{equation}
\pi_n = a_{2n+1}^2/s_{2n},		\label{eq:GL_BB4}
\end{equation}
where $a_n, s_n$ are as in Algorithm GL, and
$\pi_n$ is as in Algorithm BB4.
This observation appears to be new~-- it is not stated
explicitly in {\PAGM} or elsewhere, so far
as we know.\footnote{For example, the equivalence is not mentioned in
\cite{Bailey88}, \cite{BBB89}, 
\cite{Guillera08}, \cite{Guillera16} or \cite{Kanada88}.}

Before proving the result, we give some empirical evidence for it,
since that is how the result was discovered~-- in the spirit of
``Experimental Mathematics'', as beloved by Jon Borwein.
In Table~\ref{tab:f},
$n+1$ is the number of square roots, and
the second column is the error in the approximation given by Algorithm GL1
after $n$ iterations, or by the Algorithm BB4 after $n/2$ iterations ($n$
even). The error is the same for both algorithms (verified to $1000$ decimal
digits, not all shown).

\begin{table}[ht]
\centering
\begin{tabular}{c|l}
$n$ & \hspace*{5pt} $\pi - a_{2n+1}^2/s_{2n}$ (for Algorithm GL1) or
 $\pi - \pi_n$ (for Algorithm BB4)\\[2pt]
\hline
0 &\;\; $2.2737909121669818966095465906980480562749752399816\text{e-1}$\\
2 &\;\; $7.3762509563132989512968071098827321760295030264154\text{e-9}$\\
4 &\;\; $5.4721091456899418327485331789641785565936917028248\text{e-41}$\\
6 &\;\; $2.3085807149343902668213207343869568303303472423996\text{e-171}$\\
8 &\;\; $1.1109549335576998257002904117322306941479378545140\text{e-694}$\\
\hline
\end{tabular}
\vspace*{5pt}
\caption{Approximation error for Algorithms GL1 doubled and BB4}
\label{tab:f}
\end{table}

\pagebreak[3]
\noindent 
Using the definitions of the two algorithms,
equality for the first line of the table ($n=0$) follows from
\begin{align*}
a_1^2/s_0 = \pi_0 &= {\textstyle\frac{3}{2}}+\sqrt{2}
	= \pi - 0.227\ldots
\end{align*}
For the second line ($n=2$) we have, with $t := 2^{-1/4}$,
\[
a_3 = \frac{(t^2+2t+1+2\sqrt{2t^3+2t})}{8}
\;\text{ and }\;
s_2 = \frac{8t^3-4t^2+8t-5}{16}\,\raisecomma
\]
\vspace*{-5pt}
so
\begin{equation}
\frac{a_3^2}{s_2} = \frac{(t^2+2t+1+2\sqrt{2t^3+2t})^2}
		{4(8t^3-4t^2+8t-5)}	\,\raisedot	\label{eq:a3s2}
\end{equation}
Also, from the definition of Algorithm BB4 we find, with
\[
y_1 = \frac{1-(12\sqrt{2}-16)^{1/4}}{1+(12\sqrt{2}-16)^{1/4}}\,\raisecomma
\]
\vspace*{-5pt}
that
\begin{equation}
\pi_1 = \frac{1}{(6-4\sqrt{2})(1+y_1)^4-8y_1-8y_1^2-8y_1^3}
		\,\raisedot				\label{eq:pi1}
\end{equation}
It is not obvious that the algebraic numbers given by
\eqref{eq:a3s2} and \eqref{eq:pi1}
are identical, but it can be verified that they both
have minimal polynomial
\begin{align*}
P(x) :=& 1 - 1635840576 x - 343853312 x^2 + 60576043008 x^3\\[-5pt]
& \;\; + 1865242664960 x^4 - 16779556159488 x^5 + 37529045696512 x^6\\[-5pt]
& \;\;\;\; - 29726424956928 x^7 + 6181548457984 x^8.
\end{align*}
Using Sturm sequences~\cite{Sturm}, it may be shown that
$P(x)$ has two real roots, one in the interval 
$[0,1]$, and the other in $[3,4]$.
A numerical computation shows that $|a_3^2/s_2 - \pi_1| < 1$,
but both $a_3^2/s_2$ and $\pi_1$ are real roots of $P(x)$,
so they must be equal.

Clearly this ``brute force'' approach does not generalise.
To prove the equivalence of Algorithms BB4 and GL1,
we first consider the equivalence of Algorithms BB2 and  GL1.

\begin{theorem}				\label{thm:GL_BB2}
Algorithm BB2 is equivalent to Algorithm GL1,
in the sense that
\[\widehat{\pi}_n = a_{n+1}^2/s_n,\]
where $\widehat{\pi}_n = 1/\alpha_n$ is as in Algorithm BB2,
and $a_{n+1}, s_n$ are as in Algorithm GL.
\end{theorem}
\begin{proof}
In the proof we take $n \ge 0$, $q = e^{-\pi}$, and
assume that $a_n, b_n, c_{n+1}, s_n$ are defined as in
Algorithm GL, and $k_n, \alpha_n, \widehat{\pi}_n$ are as in Algorithm BB2.

Algorithm GL implements the recurrence
\begin{equation}
s_{n+1} = s_n - 2^n c_{n+1}^2,				\label{eq:GL_rec}
\end{equation}
whereas Algorithm BB2 implements the recurrence
\begin{equation}
\alpha_{n+1} = (1+k_{n+1})^2\alpha_n - 2^{n+2}k_{n+1}.	\label{eq:BB2_rec}
\end{equation}
We show that the recurrences \eqref{eq:GL_rec}--\eqref{eq:BB2_rec}
are related.
Noting the remark on subscripts following the statement of
Algorithm BB2, we see that
\hbox{$k_n = {c_{n+1}/}{a_{n+1}}$},
since both sides equal
$\theta_2^2(q^{2^{n+1}})/\theta_3^2(q^{2^{n+1}})$.
Thus
\begin{equation}
1+k_{n+1} = a_{n+1}/a_{n+2}. 					\label{eq:kn_an}
\end{equation}

Define $\beta_n := a_{n+1}^2\alpha_n$ and $\gamma_n := a_{n+2}^2 k_{n+1}$.
Substituting \eqref{eq:kn_an} into~\eqref{eq:BB2_rec} and
clearing the fractions gives
\begin{equation}
\beta_{n+1} = \beta_n - 2^{n+2}\gamma_n.		\label{eq:BB2_rec2}
\end{equation}
Now
\[
4\gamma_n 	= 4a_{n+2}^2 k_{n+1} 
		= 4a_{n+2} c_{n+2}
		= \theta_2^4(q^{2^{n+2}})/\theta_3^4(q)
		= c_{n+1}^2,
\]
so~\eqref{eq:BB2_rec2} is equivalent to
\begin{equation}
\beta_{n+1} = \beta_n - 2^{n}c_{n+1}^2.                \label{eq:BB2_rec3}
\end{equation}
This is essentially the same recurrence as~\eqref{eq:GL_rec}.
Also, $s_0 = 1/4$ and
$\beta_0 = a_1^2\alpha_0 = 1/4$, so $s_0 = \beta_0$.
It follows that $s_n = \beta_n$ for all $n \ge 0$.
Thus $s_n = a_{n+1}^2\alpha_n$, and
\[\widehat{\pi}_n = 1/\alpha_n = a_{n+1}^2/s_n,\] 
which completes the proof.
\Halmos
\end{proof}

\pagebreak[3]

\begin{corollary}				\label{cor:GL_BB4}
Algorithm BB4 is equivalent to Algorithm GL1 doubled, in the
sense that \[\pi_n = a_{2n+1}^2/s_{2n},\] where $\pi_n$ is as in Algorithm
BB4, and $a_{n}, s_{n}$ are as in Algorithm GL.
\end{corollary}
\begin{proof}
The Borwein brothers noted~\cite[pg.\ 171]{PAGM} that Algorithm BB4
is equivalent to \hbox{Algorithm} BB2 doubled,%
\footnote{In fact, this is how Algorithm BB4 was discovered, by doubling
Algorithm BB2 and then making some straightforward program optimisations.
}
i.e.\ $\pi_n = \widehat{\pi}_{2n}$.
Thus, the result follows from Theorem~\ref{thm:GL_BB2}.
\Halmos
\end{proof}

\begin{corollary}				\label{cor:BB4_bound}
For Algorithm BB4, the error bound~\eqref{eq:BB4_bound} holds.
\end{corollary}
\begin{proof}
In view of Corollary~\ref{cor:GL_BB4},
the error bound~\eqref{eq:BB4_bound} follows 
from~\eqref{eq:GL_BB4} and the error bound~\eqref{eq:GS_lower1} for
Algorithm~GL.
\Halmos
\end{proof}

\section[Some Fast (but Linear) Algorithms for Pi]%
{Some Fast (but Linear) Algorithms for $\pi$}
	\label{sec:Ramanujan-Sato}
	
Let $(x)_n := x(x+1)\cdots(x+n-1)$ denote the 
\emph{ascending factorial}.
In Chapter 5 of \PAGM, 
Jon and Peter Borwein discuss \emph{Ramanujan-Sato} series such as
\[	
\frac{1}{\pi} = 2^{3/2}\sum_{n=0}^\infty
 \frac{(\frac14)_n (\frac12)_n (\frac34)_n}{(n!)^3}\,
  \frac{(1103+26390n)}{99^{4n+2}}\,\raisedot
\]
This is linearly convergent, with rate $1/99^4$,
so adds nearly eight decimal digits per term,
since $99^4 \approx 10^8$.

A more extreme example is the Chudnovsky series~\cite{Chudnovsky}
\begin{equation} 
\frac{1}{\pi} = 12\sum_{n=0}^\infty
 (-1)^n\, \frac{(6n)!\,(13591409+545140134n)}{(3n)!\,(n!)^3\, 
 640320^{3n+3/2}}
				\label{eq:Chudnovsky_640320}
 \,\raisecomma
\end{equation}
which adds about $14$ decimal digits per term.

Although such series converge only linearly, their convergence is so fast
that they are competitive with higher-order algorithms such as 
Algorithm GL
for computing highly accurate
approximations to $\pi$. Which
algorithm is the fastest in practice depends on
details of the implementation and on technological factors such
as memory sizes and access times. 

\section{Fast Algorithms for the Elementary Functions}
		\label{sec:elementary-fns}

In this section, we
consider the \emph{bit-complexity} of algorithms.
The bit-complexity of an algorithm is the (worst
case) number of single-bit operations required to complete the algorithm.
For a fuller discussion, see Chapter~$6$ of \PAGM. We are interested
in asymptotic results, so are usually willing to ignore constant factors.

If all operations are performed to (approximately) the same precision, then
it makes sense to count \emph{operations} such as multiplications, divisions
and square roots. Algorithms based on the AGM fall into this category. 

If the precision of the operations varies widely, then
bit-complexity is a more sensible measure of complexity. 
An example is
Newton's method, which is self-correcting, so can be started with low
precision. Another example is summing a series with rational terms, such as
$e = \sum_{k=0}^\infty {1}/{k!}$.

The bit-complexity of multiplying two $n$-bit numbers to obtain a $2n$-bit
product is denoted by $M(n)$.  The classical algorithm shows that
$M(n) = O(n^2)$, but various asymptotically faster algorithms exist.
The best result so far, 
due to Harvey, van der Hoeven and Lecerf~\cite{Harvey16}, is
\[M(n) = O\left(n\log n\, K^{\log^*\!\!n}\right)\]
with $K=8$.
Here the \emph{iterated logarithm} function $\log^*\!n$
is defined by
\vspace*{-10pt}
\[
\log^*\!n := \begin{cases}
		0 & \text{if}\; n \le 1;\\
		1 + \log^*(\log n) & \text{if}\; n > 1.
	      \end{cases}
\]
It is unbounded but grows \emph{extremely} slowly as $n \to \infty$,
e.g.~slower than 
\[
\log\log\cdots\log n \text{ {[}for any fixed number of logs]}.
\]
Indeed, if the multiplication algorithm
is implemented on a computer that fits in the observable universe
and has components no smaller than atomic nuclei, then we can
safely assume that $\log^*\!n$ is bounded by a moderate constant, and that
multiplication has bit-complexity $O(n\log n)$.

We follow \PAGM\ 
and assume that $M(n)$ is nondecreasing and
satisfies the weak regularity condition
\[2M(n) \le M(2n) \le 4M(n).\]

Newton's method can be used to compute reciprocals and square roots
with bit-complexity
\[
O\left(M(n) + M(\lceil n/2\rceil) + 
M\left(\left\lceil n/2^2\right\rceil\right) + \cdots\ + M(1)\right)
= O(M(n)).
\]
It can be shown that the bit-complexities
of squaring, multiplication, reciprocation, division, 
and root extraction are asymptotically the same,
up to small constant factors~\cite{rpb032}.
All these operations have bit-complexity of order $M(n)$.

To compute $\pi$ to $n$ digits (binary or decimal) by the arctan
formula~\eqref{eq:arctan_3},
or to compute $1/\pi$ by the Chudnovsky series~\eqref{eq:Chudnovsky_640320},
we have to sum of order $n$ terms.
Using \emph{divide and conquer}, also called
\emph{binary splitting}~\cite{rpb032,binary-splitting},%
\footnote{Somewhat more general, but based on the same idea,
is E.~Karatsuba's \emph{FEE method}~\cite{FEE}.}
this can be done with bit-complexity
\[O(M(n)\log^2 n).\]

Suppose we compute $\pi$ to $n$-digit accuracy using one of the 
quadratically convergent AGM algorithms. This requires 
$O(\log n)$ iterations, each of which has bit-complexity
$O(M(n))$. Thus, the overall bit-complexity is
\[O(M(n)\log n).\]
This is (theoretically) better than series summation methods, the best of
which have bit-complexity of order $M(n)\log^2 n$.

In practice, a method with
bit-complexity of order $M(n)\log^2\! n$ may be faster than a method
with bit-complexity of order $M(n)\log n$ unless $n$ is sufficiently large.
This is one reason for the recent popularity of the 
Chudnovsky series~\eqref{eq:Chudnovsky_640320} for
high-precision computation of $\pi$, even though the AGM-based methods are
theoretically (i.e.~asymptotically) more efficient.

In \S\ref{sec:history},
we mentioned Salamin's algorithm for computing $\log x$
for sufficiently large $x = 4/k$, i.e.~sufficiently small $k$,
using~\eqref{eq:logAGM}.
We can evaluate $K'(k)/\pi$ using the AGM with $(a_0,b_0) = (1, k)$, 
and hence approximate
$\log(4/k)$, assuming that $\pi$ is precomputed.  
To compute $\log x$ to $n$-bit accuracy requires about
$2\log_2(n)$ AGM iterations, or $3\log_2(n)$ iterations if we count
the computation of $\pi$.

If $x$ is not sufficiently large, we can use the identity
$\log(x) = \log(2^p x) - p\log 2$, where $p$ is a sufficiently large integer
(but not too large or excessive cancellation will occur).
This assumes that $\log 2$ is precomputed, and that the precision is
increased to compensate for cancellation.

To obtain a small relative error when $x$ is close to $1$, 
say $|x-1| < 2^{-n/\log n}$, it is better to use the Taylor series
for $\log(1+z)$, with $z = x-1$. The Taylor series computation can
be accelerated by ``splitting'',
see~\cite[\S4.4.3]{rpb226} and \cite{Smith89},

The $O(k^2)$ error term in the expression~\eqref{eq:logAGM}
can be written explicitly using hypergeometric
series, see~\cite[(1.3.10)]{PAGM}.
This gives one way of improving the accuracy of the approximation
$K'(k)$ to $\log(4/k)$.
We give an alternative using theta functions, for which the series
converge faster than the hypergeometric series (which converge
only linearly).
The result~\eqref{eq:Sasaki-Kanada} follows from several 
identities given in \S\ref{sec:prelim}. We collect them here for
convenience:
\begin{align*}
\log(1/q) &= \pi K'(k)/K(k),\\
k &= \theta_2^2(q)/\theta_3^2(q),\\
K(k) &= (\pi/2)\,\theta_3^2(q),\\
K'(k) &= (\pi/2)/\AGM(1,k).
\end{align*}
Putting these pieces together gives
the elegant result of Sasaki and Kanada~\cite{Sasaki82}
\begin{equation}
\log(1/q) = \frac{\pi}{\AGM(\theta_2^2(q), \theta_3^2(q))}\,\raisedot
					\label{eq:Sasaki-Kanada}
\end{equation}
In~\eqref{eq:Sasaki-Kanada}
we can replace $q$ by $q^4$ to avoid fractional powers of $q$ in the
expansion of $\theta_2(q)$, obtaining an exact formula for
all $q\in (0,1)$:
\begin{equation}
 \log (1/q) = 
  \frac{\pi/4}{\AGM(\theta_2^2(q^4),\theta_3^2(q^4))}\,\raisedot
					\label{eq:Sasaki-Kanada4}
\end{equation}

As in Salamin's algorithm, we have to ensure that 
$x:=1/q$ is sufficiently large,
but now there is a trade-off between increasing $x$ or taking more terms
in the series defining the theta functions.
For example, to attain $n$-bit accuracy, if $x > 2^{n/36}$,
we can use
$\theta_2(q^4) = 2(q + q^9 + q^{25} + O(q^{49}))$ and
$\theta_3(q^4) = 1 + 2(q^4 + q^{16} + O(q^{36}))$.
This saves about four AGM iterations, compared to Salamin's algorithm.
We remark that a result similar to 
\eqref{eq:Sasaki-Kanada} and \eqref{eq:Sasaki-Kanada4} is
given in (7.2.5) of {\PAGM}, but with an unfortunate typo
(a reciprocal is missing).

So far we have assumed that the initial values $a_0, b_0$ in the AGM
iteration are real and positive. There is no difficulty in extending the 
results that we have used to complex $a_0, b_0$, provided that they are
nonzero and $a_0/b_0$ is not both real and negative.
For simplicity, we assume that $a_0, b_0 \in \mathcal{H} =
\{z\,|\, \Re(z) > 0\}$.

In the AGM iteration (and in the definition of the geometric mean) there is
an ambiguity of sign.  We always choose the square root with positive real
part.  
Thus the iterates $a_n, b_n$ are uniquely defined and remain in
the right half-plane $\mathcal{H}$.

When using~\eqref{eq:Sasaki-Kanada4}, 
we may need to apply a rotation to $q$,
say by a multiple of $\pi/3$, 
in order to ensure that the starting values $(\theta_2^2(q^4),
\theta_3^2(q^4))$ for the AGM lie in$\mathcal{H}$.%
\footnote{Alternatively, we could
drop the simplifying assumption that $a_0, b_0 \in \mathcal{H}$
and use the ``right choice'' of Cox~\cite[pg.~284]{Cox84} to 
implement the AGM correctly.
}

\pagebreak[3]
For $z\in\mathbb{C}\backslash\{0\}$, 
$\;\log(z) = \log(|z|) + i\arg(z)$,
provided we use the principal values of the logarithms.
Thus, if $x\in\mathbb{R}$, we can use the complex AGM to compute
\[\arctan(x) = \Im(\log(1 + ix)).\]
$\arcsin(x), \arccos(x)$ etc can be computed via $\arctan$ using
elementary trigonometric identities
such as \[\arccos(x) = \arctan(\sqrt{1-x^2}/x).\]

Since we can compute $\log, \arctan, \arccos, \arcsin$, we can
compute $\exp, \tan, \cos, \sin$
(in suitably restricted domains) 
using Newton's method.
The trigonometric functions can also be computed via the complex
exponential.
Similarly for the hyperbolic functions $\cosh, \sinh, \tanh$ and their
inverse functions.

Although computing the elementary functions via the complex AGM is
conceptually straightforward, it introduces the overhead of complex
arithmetic. It is possible to avoid complex arithmetic by the use of
Landen transformations (which transform incomplete elliptic integrals).
See exercise 7.3.2 of \PAGM\ for an outline of this approach,
and~\cite{rpb034} for more details.

Whichever approach is used, the bit-complexity of computing
$ n$-bit approximations to any of the
elementary functions ($\log, \exp, \arctan, \sin, \cos, \tan$, etc)
in a given compact set $A \subset \mathbb{C}$ that excludes singularities
of the relevant function is $ O(M(n)\log n)$.
Here ``$ n$-bit approximation'' means with {absolute} 
error bounded by~$2^{-n}$.
We could require {relative} error bounded by 
$ 2^{-n}$, but the proof would depend on a Diophantine approximation
result such as Mahler's well-known result on approximation of $\pi$
by rationals~\cite{Mahler},
because of the difficulty of guaranteeing a small relative error
in the neighbourhood of a zero of the function.\footnote{Mahler's result
is sufficient for the usual elementary functions, whose zeros are rational
multiples of $\pi$, but it is not applicable to the problem of
computing combinations of these functions,
e.g.\ $\exp(\sin x) + \cos(\log x)$, with small relative accuracy.
In general, we do not know enough about the rational approximation
of the zeros of such functions to guarantee a small \emph{relative} 
\hbox{error}.
However, the result that we stated for computing elementary functions with a
small \emph{absolute} \hbox{error} extends to finite combinations of
elementary functions under the operations of addition, multiplication,
composition, etc. Indeed, the set of \emph{elementary functions} is usually
considered to include such finite combinations, although precise definitions
vary.  See, for example,
\S7.3 of {\PAGM}, Knopp~\cite[pp.~96--98]{Knopp}, Liouville~\cite{Liouville},
Ritt~\cite{Ritt}, and Watson~\cite[pg.\ 111]{Watson}.
}

Certain non-elementary functions can be computed with
bit-complexity\linebreak
\hbox{$O(M(n)\log n)$} via the AGM.  For example, we mention complete and
incomplete elliptic integrals, elliptic functions,
and the Jacobi theta functions
$\theta_2(q), \theta_3(q), \theta_4(q)$.
Functions that appear \emph{not} to be in this class of ``easily
computable'' functions include the Gamma function $\Gamma(z)$ and the
Riemann zeta function $\zeta(s)$.

Algebraic functions can be computed with bit-complexity $O(M(n))$,
see for example \cite[Thm.~6.4]{PAGM}. 
It is plausible to conjecture that no elementary transcendental functions
can be computed with bit-complexity $O(M(n))$ (or even
$o(M(n)\log n)$). However, as usual in complexity theory, nontrivial lower
bounds are difficult to prove and depend on the precise model of
computation.

\pagebreak[3]

\begin{samepage}
\subsection*{Acknowledgement}

I am grateful to Jon and Peter Borwein for becoming sufficiently interested
in this subject to write their book {\PAGM} only a few years after the
publication of \cite{rpb028,rpb032,rpb034,Salamin76}. Reading a copy of
{\PAGM} was my first introduction to the Borwein brothers, and was the start
of my realisation that we shared many common interests, despite living in
different hemispheres. Much later, after Jon and his family 
moved to Newcastle (NSW), I followed
him, bringing our common interests closer together, and benefitting
from frequent interaction with him.
\end{samepage}

Thanks are also due to David Bailey for his assistance, and to the Magma
group for their excellent software~\cite{Magma1}.

The author was supported in part by an Australian Research Council 
grant\linebreak
DP140101417. Jon Borwein was the Principal Investigator on this grant,
which was held by Borwein, Brent and Bailey.

\pagebreak[3]
\bibliographystyle{}
\bibliography{}

\end{document}